\documentclass[final,12 pt, 3p]{elsarticle}
 \usepackage{graphics}
 \usepackage{verbatim }
 \usepackage{graphicx}
 \usepackage{epsfig}
\usepackage{amssymb}
\usepackage{amsmath}
\usepackage{showlabels}
\usepackage{amsthm}

   \newtheorem{theorem}{Theorem}[section]
   
   \newtheorem{lemma}[theorem]{Lemma}
   \newtheorem{corollary}[theorem]{Corollary}

   \newtheorem{definition}[theorem]{Definition}

   \newtheorem{remark}[theorem]{Remark}

\let\today\relax
\journal{Linear Algebra and Its Applications}

\begin{document}

\begin{frontmatter}

\title{Construction of Lorenz Cone with Invariant Cone Using Dikin Ellipsoid for Dynamical Systems
}
\author{Yunfei Song\footnote{Email: songyunfei1986@gmail.com}}
%\date{July 2021}
\address{Quant Strategies Group, BA Securities\\
\large\emph{July 11, 2021}}

\begin{abstract}
%% Text of abstract
In this paper, some special Lorenz cones are constructed using Dikin ellipsoid and some hyperplane. We also study the structure of the constructed cones, especially the eigenvalues structure of the related matrix in the formula of elliposid. These novel Lorenz cones  which locate in positive orthant by its construction are potential candidates to design invariant cone for a given dynamical system.  It provides more flexibility for practitioner  to choose more cones in the application for system stability analysis.   

\end{abstract}


\end{frontmatter}

\section{Introduction}
The history of the analysis of invariant properties for dynamical
system goes back 1890s, when mathematician Lyapunov firstly
introduced his theory  on ordinary differential equations. This theory was
later on named by Lyapunov theory or stability theory. This theory proves
that the stability of a dynamical system, i.e. all trajectories will
approach a fixed point as time is going to infinity, can be
transformed to analyze the properties of a function that is named
Lyapunov candidate function. Stability theory is related to the concept of invariant set. 
Blanchini \cite{Blanchini} provides an excellent survey paper about
invariance of dynamical system. Positively invariant set is an important concept that is widely used in 
many areas e.g., control theory, electronic systems, economics, etc. e.g., see \cite{Boyd, luen, shen}. Given a set and a dynamical system,
verifying whether the set is an invariant set of the given system is an interesting topic in this field. A general equivalent condition is given 
by Nagumo, e.g., \cite{nagu}. The explicit conditions for linear system and some common sets are derived by Hov\'{a}th, et \cite{bits1, bits2, song1, song2}. 
The discrete and continuous system are usually considered separately for invariant sets.  
Preserving the 
invariance from a continuous system to a discrete system  by using certain discretization methods is studied by Hov\'{a}th et \cite{song3}.   

In this paper, we construct some special Lorenz cone using Dikin ellipsoid and hyperplane, and the structures of the constructed cones are studied.  The novelty 
of this method is that linking the mathematical optimization tools to the invariant set. The  motivation of the construction is to design more potential invariant cone 
within positive orthant, which is usually a common requirement in practical application.

\emph{Notation and Conventions.} We use the following
notation and conventions to avoid unnecessary repetitions
\begin{itemize}
  \item The inertia of a matrix is denoted by inertia$\{Q\}=\{a,b,c\}$ that indicates the number
of positive, zero, and negative eigenvalues of $Q$, respectively.
  \item The basis in $\mathbb{R}^n$ is denoted by $e_1=\{1,0,...,0\}, e_2=\{0,1,...,0\},...,
  e_n=\{0,0,...,1\}$. And we let $e=(1,1,...,1).$
  \item The aim of using $x_{[k]}$ to indicate the discrete
state variable is to distinguish with the $k$-th coordinate, denoted
by $x_k$, of variable $x$.
\end{itemize}

\section{Preliminaries}
\subsection{Invariant Sets}
In this paper, the linear discrete and continuous 
systems are described as follows:
\begin{equation}\label{eqn:dy2}
x_{k+1}=A_dx_{k},
\end{equation}
\begin{equation}\label{eqn:dy1}
\dot{x}(t)=A_cx(t),
\end{equation}
where $x_{k}, x_{k+1}, x(t)\in\mathbb{R}^{n}$
are the state variables, and $A_d, A_c\in \mathbb{R}^{n\times n}$ are constant coefficient 
matrices.

Followed on the linear systems above, the invariant sets for the corresponding discrete and continuous forms are introduced. 

\begin{definition}\label{definv1}
A set $\mathcal{S}\in\mathbb{R}^n$ is called an \textbf{invariant
set} for
the discrete system (\ref{eqn:dy2}) if $x_{k}\in \mathcal{S}$ implies
$x_{k+1}\in \mathcal{S}$, for all $k\in \mathbb{N}.$
A set $\mathcal{S}\in\mathbb{R}^n$ is called an \textbf{invariant
set} for
the continuous system (\ref{eqn:dy1}) if $x(0)\in \mathcal{S}$ implies
that $x(t)\in\mathcal{S}$, for all $t\geq0$.
\end{definition}

In fact, the
set $\mathcal{S}$ in Definition \ref{definv1} is usually refereed to as
\emph{positively invariant set} in literatures, as we can see that it only considers the forward time or nonnegative time. Since we only consider positively
invariant set in this paper, we call it invariant set for
simplicity. 

\begin{definition}\label{definv123}
An operator $\mathcal{A}$ is called \textbf{invariant} on a set $\mathcal{S}\in\mathbb{R}^n$ if $\mathcal{A}\mathcal{S}\subset\mathcal{S}.$ 
\end{definition}

An immediate connection between Definition \ref{definv1} and \ref{definv123} is that a set $\mathcal{S}$ is an invariant
set for the linear discrete  system (\ref{eqn:dy2}) if and only if
$A_d$ is invariant on
$\mathcal{S}$;
$\mathcal{S}$ is an invariant
set for the linear continuous  system (\ref{eqn:dy1}) if and only if \footnote{Here recall that
$e^{At}=\sum_{k=1}^\infty\frac{(At)^k}{k!}$.} $e^{A_ct}$
 is invariant on
$\mathcal{S}$.

\subsection{Hyperplane,  Ellipsoid, Lorenz Cone, and Dikin Ellipsoid}
In this subsection, the definitions  and formulas of some common types in $\mathbb{R}^n$, namely, hyperplane,
ellipsoid, Lorenz cone, and Dikin ellipsoid are introduced. 

\begin{definition}\label{defhyp}
A \textbf{hyperplane}, denoted by $\mathcal{H}\in \mathbb{R}^n,$ is
represented as either
\begin{equation}\label{eqhyp1}
\mathcal{H}=\mathcal{H}(a,\alpha)=\{x\in \mathbb{R}^n\,|\,a^Tx=\alpha,
\alpha\in \mathbb{R}\},
\end{equation}
 or
equivalently
\begin{equation}\label{eqhyp2}
 \mathcal{H}=\mathcal{H}(x_0,H)=\{x\in \mathbb{R}^n\,|\,x=x_0+Hz, z\in \mathbb{R}^{n-1}\},
\end{equation}
 where $x_0$ is a point in $\mathcal{H}$ and $H$ consists of all  basis, denoted by $h_1,h_2,...,h_{n-1},$ of the
complementary space of $a$,  i.e., $H=[h_1,h_2,...,h_{n-1}]\in
\mathbb{R}^{n\times n-1}, a^TH=0, $ $ H^Ta=0,$ and
\emph{span}$\{a,H\}=\mathbb{R}^n.$
\end{definition}

The vector $a$ in (\ref{eqhyp1}) is called the \emph{normal vector} of the hyperplane $\mathcal{H}.$
The matrix $H$ in (\ref{eqhyp2}) is called the \emph{complementary matrix}  of the vector $a$. Moreover, if 
$h_1,h_2,...,h_{n-1}$ are mutually orthonormal, i.e., $h_i^Th_j=\delta_{ij}$, where $\delta_{ij}$ is the Kronecker delta, then we have $H^TH=I_{n-1},$ and we call
$H$ the \emph{orthonormal complementary matrix} of the vector $a$.
Formule (\ref{eqhyp1}) is straightforward and can be seen in many literatures. We will use Formula (\ref{eqhyp2}) in this paper since it  considers a hyperplane in the affine plane.

\begin{comment}
\begin{definition}\label{defheep}
An \textbf{ellipsoid}, denoted by $\mathcal{E}\in \mathbb{R}^n$, centered at the
origin, is represented as follows:
\begin{equation}\label{elli}
\mathcal{E}=\{x\in\mathbb{ R}^n ~|~ x^TQx\leq 1\},
\end{equation}
where  $Q\in \mathbb{R}^{n\times n}$ and $Q\succ0$. 
\end{definition}
Note that any ellipsoid with nonzero center can be easily mapped to an ellipsoid centered at origin.

\begin{definition}\label{defhyloc}
A \textbf{Lorenz cone}\footnote{A Lorenz cone is also called an
ice cream cone, or a second order cone.}, denoted by $\mathcal{C_L}\in \mathbb{R}^n$,
with vertex at origin,  is represented as follows:
\begin{equation}\label{ellicone}
\mathcal{C_L}=\{x\in \mathbb{R}^n~|~x^TQx\leq 0,~ x^TQu_n\leq0\},
\end{equation}
where $Q\in \mathbb{R}^{n\times n}$ is a symmetric nonsingular
matrix with one negative eigenvalue $\lambda_n$, i.e., ${\rm
inertia}\{Q\}=\{n-1,0,1\}$. 
\end{definition}
Similar to ellipsoids, any Lorenz cone with nonzero vertex can be easily mapped to a Lorenz cone with vertex at origin.
In particular, let $Q=\tilde{I}=\text{{diag}}\{1,...,1,-1\}$ and $e_n=(0,...,0,1)^T$,
then we have Lorenz cone $ \mathcal{C_L^*}=\{x\in \mathbb{R}^n\;|\; x^T\tilde{I}x\leq 0,
x^Te_n\geq0\}.$ We call $\mathcal{C_L^*}$ the  \emph{standard Lorenz cone}. In fact, one can prove that each Lorenz cone can be mapped into 
a standard Lorenz cone by certain transformation, e.g., \cite{song1}. 

\end{comment}

%\begin{comment}
\begin{definition}\label{def4}
An \textbf{ellipsoid}, denoted by $\mathcal{E}\in\mathbb{R}^n, $ is represented as
\begin{equation}\label{eq16}
\mathcal{E}=\mathcal{E}(Q,p,\rho)=\{x\in \mathbb{R}^n |
x^TQx+2p^T+\rho\leq1\},
\end{equation}
where $Q\succ0,$ and $ \rho = p^TQ^{-1}p$. A \textbf{standard
ellipsoid}, denoted by $\mathcal{E}^*\in \mathbb{R}^n$,  is
represented as
\begin{equation}\label{eq17}
\mathcal{E}^*=\mathcal{E}(\tilde{Q},\textbf{0},0)=\{x\in
\mathbb{R}^n|a_1x_1^2+a_2x_2^2+...+a_nx_n^2\leq 1\}=\{x\in
\mathbb{R}^n|x^T\tilde{Q}x\leq1\}
\end{equation}\label{eq21}
where $\tilde{Q}=\emph{diag}\{a_1,a_2,...,a_n\}, $ with $ a_i>0$,
for $i=1,2,...,n$.
\end{definition}

\begin{definition}\label{def3}
A \textbf{Lorenz cone}, denoted by $\mathcal{C_L}\in \mathbb{R}^n,$ is represented
as
\begin{equation}\label{eq13}
\mathcal{C_L}=\mathcal{C_L}(Q,p,\rho)=\{x\in
\mathbb{R}^n|x^TQx+2p^Tx+\rho\leq0\},
\end{equation}
where $Q$ is a symmetric matrix with inertia$(Q)=\{n-1,0,1\},$ $p\in
\mathbb{R}^n,$ and $ \rho = p^TQ^{-1}p$. A \textbf{standard Lorenz
cone}, denoted by $\mathcal{C_L}^*\in \mathbb{R}^n,$ is represented
as
\begin{equation}\label{eq14}
\mathcal{C_L^*}=\{x\in \mathbb{R}^n|x_1^2+x_2^2+\cdots+x_{n-1}^2\leq
x_n^2,  x_n\geq0\} =\{x\in \mathbb{R}^n|x^T \tilde{I} x\leq0,
x^T\tilde{I} e_n\leq0\}
\end{equation}
where $\tilde{I}=\emph{diag}\{1,1,...,1,-1\}, $ and $
e_n=(0,...,0,1)^T.$
\end{definition}

\begin{remark}
The center of $\mathcal{E}$ in the form of (\ref{eq16}) is
$-Q^{-1}p.$ The vertex of $\mathcal{C_L}$ is $-Q^{-1}p$, and  the
axis of $\mathcal{C_L}$ in the form of (\ref{eq13}) is $\{x\in
\mathbb{R}^n|x=-Q^{-1}p+\alpha Pe_n \}$, where $\alpha \in
\mathbb{R}$.
\end{remark}

In fact, one can prove that each Lorenz cone can be mapped into 
a standard Lorenz cone by certain transformation, e.g., \cite{song1}. We notice that a Lorenz cone in the form of (\ref{eq13}) consists of
two branches, one of which is centrosymmetric to the other one with
respect to the vertex. A standard Lorenz cone $\mathcal{C_L^*} $ in
the form of (\ref{eq14}) is a convex set and a self-duel cone, i.e.,
the dual cone\footnote{A duel cone of a cone $\mathcal{C}$ is
defined as $\{y\in \mathbb{R}^n~|~y^Tx\geq 0, \text{ for all } x\in
\mathcal{C}\}.$} is coincidence with itself. Also, the formula of
$\mathcal{C_L^*}\cup (-\mathcal{C_L^*})$ is given as $\{x\in
\mathbb{R}^n|x^T \tilde{I} x\leq0\}.$

The relationships between general and standard ellipsoids and
between general and standard Lorenz cones are given in the following
lemma.

\begin{lemma}
There exists two nonsingular matrices $P, $ and $\tilde{P}$, such
that
\begin{equation}\label{eq18}
\tilde{P}^{-1}\mathcal{E}^*=P^{-1}\mathcal{E}+(QP)^{-1}p,
\end{equation}
where $\mathcal{E}$ and $\mathcal{E}^*$ are  defined as (\ref{eq16})
and (\ref{eq17}), respectively. There exists a nonsingular matrix
$\bar{P}$, such that
\begin{equation}\label{eq15}
\mathcal{C_L^*}\cup
(-\mathcal{C_L^*})=\bar{P}^{-1}\mathcal{C_L}+(Q\bar{P})^{-1}p,
\end{equation}
where $\mathcal{C_L}$ and $\mathcal{C_L^*}$ are given as
(\ref{eq13}) and (\ref{eq14}), respectively.
\end{lemma}
\begin{proof} We only present the proof of (\ref{eq18}), the proof of (\ref{eq15})
is analogous.  Since $Q\succ0$, there exists an orthogonal matrix
$U$ consisting of all the eigenvectors $\{\lambda_i\}$ of $Q$, such
that $U^TQU=\text{diag}\{\lambda_1,...,\lambda_n\}$. Denoting
$Q_1=\text{diag}\{\frac{1}{\sqrt{\lambda_1}},...,\frac{1}{\sqrt{\lambda_n}}\}$,
we have  $Q_1^TU^TQUQ_1=\tilde{I}.$ We let $P=UQ_1$ and
$\tilde{P}=\text{diag}\{\sqrt{a_1},...,\sqrt{a_n}\}$, both of which
are nonsingular, then (\ref{eq16}) and (\ref{eq17}) can be
respectively rewritten as
\begin{equation*}
\mathcal{E}=\{x\in
\mathbb{R}^n~|~(P^{-1}x+\tilde{I}P^{T}p)^T(P^{-1}x+\tilde{I}P^{T}p)\leq1\},
 \text{ and } \mathcal{E}^*=\{x\in
\mathbb{R}^n~|~(\tilde{P}^{-1}x)^T(\tilde{P}^{-1}x)\leq1\}.
\end{equation*}
Noting that $P\tilde{I}P^{T}=Q^{-1}$ implies
$\tilde{I}P^{T}=(QP)^{-1}$, we deduce (\ref{eq18}) immediately.
\end{proof}
%\end{comment}

\begin{definition} \label{dikin}
\cite{bert, Terlaky}
A \textbf{Dikin Ellipsoid}, denoted by $\mathcal{E_D}\in
\mathbb{R}^n$,  is represented as
\begin{equation}\label{eq31}
\mathcal{E_D}=\Big\{x\in
\mathbb{R}^n\,|\,\sum_{i=1}^n\frac{(x_i-c_i)^2}{c_i^2}\leq1\Big\}=\{x\in
\mathbb{R}^n\,|\,(x-c)^TC^{-2}(x-c)\leq 1\},
\end{equation}
where $c=(c_1,c_2,...,c_n)^T$,
$C=\emph{diag}\{c_1,c_2,...,c_n\}$, and $c_i>0$, for $i=1,2,...,n.$
\end{definition}

The point $c$ is the center of Dikin elliposid according to (\ref{eq31}). In fact, the ellipsoid (\ref{eq31}) was introduced by Dikin and widely used in designing some mathematical optimization 
algorithms, e.g., affine scaling interior
point methods \cite{bert,Terlaky}. A common property of every Dikin ellipsoid is
that it is constantly in the positive octant of $\mathbb{R}^n$ including its boundary. This is a key property to design mathematical optimization algorithms. 

\begin{lemma} \cite{bert,Terlaky}
Assume the Dikin ellipsoid $\mathcal{E_D}$ is given as (\ref{eq31}) and let $x\in \mathcal{E_D}$, then $x\geq0$.
\end{lemma}
\begin{proof} For every $i$, we have $\frac{(x_i-c_i)^2}{c_i^2}\leq
\sum_{i=1}^n\frac{(x_i-c_i)^2}{c_i^2}\leq1$. This yields $-c_i\leq
x_i-c_i\leq c_i,$ i.e., $0\leq x_i\leq 2c_i. $
\end{proof}

Let us consider the hyperplane $\mathcal{H}$  given as (\ref{eqhyp1}) with the normal vector
$a$, and assume $\mathcal{H}$ intersects through the center of the Dikin ellipsoid
$\mathcal{E_D}$ given as (\ref{eq31}), then we have
\begin{equation}\label{eq112}
\mathcal{H}=\mathcal{H}(a,a^Tc)=\{x\in \mathbb{R}^n\,|\,a^Tx=a^Tc\},
\end{equation}
and the intersection $\mathcal{H}\cap \mathcal{E_D}$ is also an ellipsoid.

\begin{lemma}
Let a hyperplane $\mathcal{H}$ and a Dikin ellipsoid $\mathcal{E_D}$
be given as (\ref{eq112}) and (\ref{eq31}), respectively. Then
$\mathcal{H}\cap \mathcal{E_D}=\{x\in \mathbb{R}^n\,|\,x=c+Hz\}$, where
$ z\in \mathbb{R}^{n-1}$ satisfies
$z^TH^TC^{-2}Hz\leq1, $ and $H$ is a
complementary matrix of the vector $a.$
\end{lemma}
\begin{proof}
According to the second formula of $\mathcal{H}$ given as
(\ref{eqhyp2}), and let $x_0=c,$ we have $x=c+Hz$.  
Substituting  $x=c+Hz$ into (\ref{eq31}), this lemma
is immediate.
\end{proof}

\begin{definition}
\cite{wilk} A matrix $D\in \mathbb{R}^{n\times n}$ is called an
\textbf{arrowhead matrix} if it has the following form
\begin{equation}\label{eq8}
D=\left[
  \begin{array}{cc}
    \alpha& p \\
    p^T & B \\
  \end{array}
\right],
\end{equation}
where $\alpha\in \mathbb{R},  p\in\mathbb{R}^{n-1},$ and $
B=\emph{diag}\{b_1,b_2,...,b_{n-1}\}.$  Here we assume $b_1\geq b_2\geq
...\geq b_{n-1}.$
\end{definition}

\begin{lemma}\label{lemma11}
\cite{wilk} The following properties of the arrowhead matrix $D$
given as (\ref{eq8}) are true:
\begin{enumerate}
  \item The characteristic polynomial of $D$ is
  \begin{equation}\label{eq11}
\det(\lambda
I-D)=(\lambda-\alpha)\prod_{k=1}^{n-1}(\lambda-b_k)-\sum_{j=1}^{n-1}|p_j|^2\prod_{k=1,k\neq
j}^{n-1}(\lambda-b_k).
\end{equation}
   \item The eigenvalues of $D$ are real and satisfying the following condition
\begin{equation}\label{eqn:arr}
\lambda_1\geq b_1\geq\lambda_2\geq b_2\geq \cdots\geq b_{n-1}\geq
\lambda_n.
\end{equation}
\end{enumerate}
\end{lemma}

\section{Construction of Novel Lorenz Cones}

In this section, we will construct some novel standard Lorenz cones and derive the corresponding explicit formulas for some special cases.   In particular, these novel standard Lorenz cones are constructed
by using a Dikin ellipsoid as its base\footnote{A set $\mathcal{B}$ is refereed to as  a base of a cone $\mathcal{C}$ if for any $x\in \mathcal{C}$, there exists a $\hat{x}\in \mathcal{B}$, such that $x=\lambda \hat{x}$ for some $\lambda>0$.}, or using the intersection of a Dikin ellipsoid
and a hyperplane as its base. Since the Lorenz cones we considered are standard, i.e., the vertices are the origin,  we have that the cones are constantly in the positive octant in $\mathbb{R}^n$ including its boundary. Also, the properties of the constructed Lorenz cones, especially the structures of the eigenvalues of the matrices involved in the cone formulas, are studied. The motivation of these novel cones is that they are considered as
candidate invariant sets for dynamical systems  in the positive orthant.

\begin{definition}
\cite{Laub} Let $X\in
\mathbb{R}^{m\times n}$. The \textbf{vectorization} of $X$ , denoted by $\emph{vec}(X)$, is writing all columns
of $X$ into a single vector, i.e.,
$\emph{vec}(X)=(x_{11},...,x_{m1},x_{12},...,x_{m2},...,x_{1m},...,x_{mn})^T.$
\end{definition}

\begin{lemma}
Let $\{v_i\}_{i=1}^n$ be a basis of
$\mathbb{{R}}^n$. Then  $\emph{vec}(v_iv_j^T),  1\leq i, j\leq n$ is a
basis of $\mathbb{R}^{n^2}.$
\end{lemma}
\begin{comment}
\begin{proof}
To proof this lemma is equivalent to prove that $\{v_iv_j^T\}$ is a
basis for $R^{n\times n}$ (note that $\mathbb{R}^{n^2}$ means
$\mathbb{R}^{n^2\times 1}$). It is easy to have that $\{e_ie_j^T\}$
is a basis for $R^{n\times n}$. Since $\{v_i\}$ is a basis of
$\mathbb{R}^n$, there exists a nonsingular matrix $P\in
\mathbb{R}^{n\times n}$, such that $Pv_i=e_i$, for $i=1,...,n.$ Now
let $\sum_{i,j}\lambda_{ij}v_iv_j^T=\textbf{0}\in
\mathbb{R}^{n\times n},$ we have
$\sum_{i,j}\lambda_{ij}Pv_iv_j^TP^T=
\sum_{i,j}\lambda_{ij}e_ie_j^T=\textbf{0},$ which implies that
$\lambda_{ij}=0$, for $i,j=1,...,n$. Also, for any $X\in
\mathbb{R}^{n\times n},$ we let $\tilde{X}=P^{-1}XP^{-T}$, then we
have
$X=P\tilde{X}P^T=P(\sum_{i,j}\tilde{x}_{ij}e_ie_j^T)P^T=\sum_{i,j}\tilde{x}_{ij}v_iv_j^T.$
This proof is complete.
\end{proof}
\end{comment}

\begin{corollary}\label{cor1}
Let $a\in \mathbb{{R}}^n$, and $H=[h_1,h_2,...,h_{n-1}]\in
\mathbb{R}^{n\times(n-1)}$ be a complementary matrix of the vector $a$. Then
the following vectors are a basis of $\mathbb{R}^{n^2},$
\begin{equation}
\emph{vec}(aa^T),~\emph{vec}(ah_i^T), ~\emph{vec}(h_ia^T), \text{ and } \emph{vec}(h_ih_j^T), \text{ where }
1\leq i,j\leq n-1.
\end{equation}
\end{corollary}

\begin{comment}
\begin{definition}
\cite{Laub}  Let $M\in \mathbb{R}^{p\times q}, $  $N\in
\mathbb{R}^{m\times n}.$ Then the \textbf{Kronecker product} of $M$
and $N$ denoted by $M\otimes N$ is defined as
\begin{equation*}
M\otimes N =\left[
              \begin{array}{ccc}
                m_{11}N & \cdots & m_{1q}N \\
                \vdots & \ddots & \vdots \\
                m_{p1}N & \cdots & m_{pq}N \\
              \end{array}
            \right]\in \mathbb{R}^{pm\times qn}.
\end{equation*}
\end{definition}
\end{comment}

\begin{lemma}\label{lemma1}
Let $a\in \mathbb{{R}}^n$, and $H=[h_1,h_2,...,h_{n-1}]\in
\mathbb{R}^{n\times(n-1)}$ be a complementary matrix of the vector $a$. Assume
$H^TXH=\textbf{0}\in\mathbb{R}^{(n-1)\times (n-1)},$ where $X\in
\mathbb{R}^{n\times n},$ then there exist $\mu\in\mathbb{R},$ and
$z_1, z_2\in \mathbb{R}^{n-1}$, such that
\begin{equation}\label{lem1:eq0}
X=\mu aa^T+az_1^TH^T+Hz_2a^T.
\end{equation}
If $X$ is further assumed to be symmetric, then $z_1=z_2.$
\end{lemma}

\begin{proof}
According to \cite{Laub},  the equation $H^TXH=\textbf{0}$ can be solved by the following equation
\begin{equation}\label{lem1:eq1} 
(H^T\otimes
H^T)\text{vec}(X)=\text{vec}(H^TXH)=\text{vec}(\textbf{0})\in\mathbb{R}^{(n-1)^2},
\end{equation} 
where $M\otimes N$ is the Kronecker product, e.g., \cite{Laub}.
According to (\ref{lem1:eq1}), we have that $\text{vec}(X)$ is in $\text{ker}(H^T\otimes H^T)$, i.e., the
kernel of $H^T\otimes H^T$.  We note that  $\text{vec}(h_ih_j^T),$ for $ i,j=1,...,n-1,$ are in the range space of 
$H^T\otimes H^T.$ Then,  according to Corollary \ref{cor1}, we have
that $\text{vec}(aa^T), \text{vec}(ah_i^T), $ and $\text{vec}(h_ia^T), $ for $i=1,...,
n-1$, are in  $\text{ker}(H^T\otimes H^T).$ Thus
$\text{vec}(X)$ can be represented as a linear combination of
$\text{vec}(aa^T),\text{vec}(ah_i^T), $ and $\text{vec}(h_ia^T).$
Then (\ref{lem1:eq0}) is immediate  by writing this linear
representation in a matrix form. If $X$ is a symmetric matrix, then we have $a(z_1-z_2)^TH^T=H(z_1-z_2)a^T$. By the linear independence of  $\text{vec}(ah_i^T), $ and $\text{vec}(h_ia^T),$ for $i=1,2,...,n-1,$ we have $z_1=z_2.$
\end{proof}

\begin{lemma}\label{lemma2}
Let a hyperplane $\mathcal{H}$ and a Dikin ellipsoid $\mathcal{E_D}$
be given as (\ref{eq112}) and (\ref{eq31}), respectively, and the base of 
a standard Lorenz cone $\mathcal{C_L}$ be
$\mathcal{H}\cap\mathcal{E_D}$.  Let $\mathcal{C_L}$ be represented as
$\mathcal{C_L}=\{x\,|\,x^TQx\leq0\}$. Then the following conditions
hold:
\begin{equation}\label{cond}
H^TDH=\beta H^TQH, ~~c^TQH=0, ~~c^TQc=-\beta,
\end{equation}
where $\beta$ is a positive number and $H$ is a  complementary
matrix of $a$.
\end{lemma}
\begin{proof} We use the representation 
$\mathcal{H}=\mathcal{H}(c,H)=\{x\,|\,x=c+Hz\}$. Substituting $x=c+Hz$
into (\ref{eq31}) and $\mathcal{C_L}=\{x\,|\,x^TQx\leq0\}$,
respectively, we have
\begin{equation}\label{ineq}
z^TH^TDHz-1\leq0 \text{ ~ and  ~} z^TH^TQHz+2c^TQHz+c^TQc\leq0.
\end{equation}
The lemma is immediate by comparing the two inequalities in (\ref{ineq}).
\end{proof}

\begin{theorem}\label{themmain}
Assume the conditions in (\ref{cond}) hold, then
\begin{equation}\label{eq6}
\gamma
Q=D-\frac{1+c^T\tilde{D}c}{(c^Ta)^2}aa^T-\frac{1}{c^Ta}(ac^TD\tilde{H}+\tilde{H}Dca^T),
\end{equation}
where $\gamma$ is a positive number,
$\tilde{D}=D-(D\tilde{H}+\tilde{H}D)$, and
$\tilde{H}=H(H^TH)^{-1}H^T$.
\end{theorem}

\begin{proof}
Since $\mathcal{C_L}$ is defined as
$\{x\,|\,x^TQx\leq0\}$, we  let $\beta=1$ in (\ref{cond}) for
simplicity. According to Lemma \ref{lemma1} and the first condition
in (\ref{cond}), there exist $\mu\in\mathbb{R}$ and $ z\in
\mathbb{R}^{n-1},$ such that
\begin{equation}\label{eq1}
Q=D+\mu aa^T+az^TH^T+Hza^T.
\end{equation}
If we substitute (\ref{eq1}) into the second condition in
(\ref{cond}) and note that $H^Ta=0$, $a^TH=0$, and $c^Ta=\alpha$,
then we have
\begin{equation}\label{eq2}
z=-\frac{1}{\alpha}(H^TH)^{-1}H^TDc.
\end{equation}
 If we
substitute (\ref{eq1}) and (\ref{eq2}) into the third condition in
(\ref{cond}), then we have
\begin{equation}\label{eq3}
\mu=-\frac{1}{\alpha^2}(1+c^TDc-c^T(D\tilde{H}+\tilde{H}D)c).
\end{equation}
where $\tilde{H}=H(H^TH)^{-1}H^T.$

We then substitute (\ref{eq2}) and (\ref{eq3}) into (\ref{eq1}), and this theorem is
immediate.
\end{proof}

It is interesting to note that (\ref{eq2}) is the same as the solution of
$Hz=-\frac{1}{\alpha}Dc$ solved by the least squares method. Since
$\tilde{H}$ is a symmetric matrix, all the eigenvalues of
$\tilde{H}$ are real numbers. Note that $\tilde{H}^k=\tilde{H}$
holds for any positive integer $k,$ thus the eigenvalues of
$\tilde{H}$ are either $0$ or $1$ by choosing $k=2$. Also, according
to Sylvester's law of inertia \cite{horn}, we have
inertia$\{\tilde{H}\}=\{n-1,1,0\} $ because of
$H^T\tilde{H}H=I_{n-1}.$ Thus, the matrix $\tilde{H}$ has
eigenvalues 1 with multiplier $n-1$ and 0 with multiplier 1.

\begin{remark}
The vectors $h_1,h_2,...,h_{n-1}$ in Lemma \ref{lemma1} are not
necessarily orthonormal. However, if we use the orthonormal basis,
i.e., $H^TH=I_{n-1},$ the previous computation may be simplified.
\end{remark}

We now consider the  Lorenz cones generated by the origin and
the intersection between Dikin ellipsoids and some special hyperplanes.

\subsection{ {The hyperplane is $\mathcal{H}_i=\mathcal{H}(e_i,
c_i)=\{x\,|\,e_i^Tx=c_i\}.$}}

\begin{theorem}
Let a Dikin ellipsoid $\mathcal{E_D}$ be given as
(\ref{eq31}), and a hyperplane be $\mathcal{H}_i=\mathcal{H}(e_i,
c_i)$, and the base of the standard Lorenz cone $\mathcal{C_L}$ be $\mathcal{H}_i\cap\mathcal{E_D}$. Let $\mathcal{C_L}$ be represented as
$\mathcal{C_L}=\{x\,|\,x^TQ_ix\leq0\}$. Then
\begin{equation}\label{eq7}
Q_i=D+\frac{n-3}{c_i^2}E_{ii}+\sum_{j=1,j\neq
i}^n\frac{-1}{c_ic_j}(E_{ij}+E_{ji}),
\end{equation}
where $E_{ij}$ denotes an $n\times n$  matrix that has $ij$-th entry 1 and other
entries 0.
\end{theorem}
\begin{proof}
Note that the normal vector of $\mathcal{H}_i$ is $e_i$,  thus the
complementary matrix of $e_i$ can be chosen as
$H=[e_1,...,e_{i-1},e_{i+1},...,e_n].$ Then we have
$\tilde{H}=HH^T=I_n-E_{ii}.$ Thus, we have
\begin{equation}\label{eq512}
\alpha = c_i, ~aa^T=E_{ii},~ ac^T=\sum_{j=1}^nc_jE_{ij},~
ca^T=\sum_{j=1}^nc_jE_{ji}, \text{ and }D\tilde{H}=\tilde{H}D=D-DE_{ii}.
\end{equation}
Substituting (\ref{eq512}) into (\ref{eq6}), the lemma is immediate.
\end{proof}

For example, let $i=1$ in (\ref{eq7}), we have that  
$Q_1$ is explicitly given as follows:
\begin{equation}\label{eqq1}
Q_1=\left[
    \begin{array}{cccc}
      \frac{n-2}{c_1^2} & -\frac{1}{c_1c_2} & \ldots & -\frac{1}{c_1c_n} \\
      -\frac{1}{c_1c_2} & \frac{1}{c_2^2} &  &  \\
      \vdots &  & \ddots &  \\
      -\frac{1}{c_1c_n} &  &  & \frac{1}{c_n^2} \\
    \end{array}
  \right].
\end{equation}

Note that $Q_1$ is an arrowhead matrix. For any $Q_i$  given as (\ref{eq7}), by certain row and column transformation, it can be represented as an arrowhead matrix. In fact, the matrix $Q_i$ given as (\ref{eq7}) has many interesting properties.

\begin{lemma}\label{lemma12}
For every $Q_i$ given as (\ref{eq7}), we have $\det(Q_i)=-\prod_{k=1}^n\frac{1}{c_k^2}$.
\end{lemma}
\begin{proof} We only consider  $Q_1$ given as (\ref{eqq1}). By choosing $\lambda=0$ and $A=Q_1$ in
(\ref{eq11}), we have
\begin{equation*}
\begin{split}
\det(-Q_1)
&=-\frac{n-2}{c_1^2}\prod_{k=2}^n\frac{-1}{c_k^2}-\sum_{j=2}^n\frac{1}{c_1^2c_j^2}\prod_{k=2,k\neq
j}^n\frac{-1}{c_k^2}
=(-1)^n\frac{n-2}{c_1^2}\prod_{k=2}^n\frac{1}{c_k^2}-(-1)^n\frac{n-1}{c_1^2}\prod_{k=2}^n\frac{1}{c_k^2}\\
&=(-1)^{n-1}\frac{1}{c_1^2}\prod_{k=2}^n\frac{1}{c_k^2}
=(-1)^{n-1}\prod_{k=1}^n\frac{1}{c_k^2}
\end{split}
\end{equation*}
Noting that that $\det(Q_1)=(-1)^n\det(-Q_1)$, this lemma is
immediate.
\end{proof}

\begin{lemma}
For every $Q_i$ given as (\ref{eq7}),  we have $\emph{inertia}(Q_i)=\{n-1,0,1\}$.
\end{lemma}

\begin{proof}
Without loss of generality, we only consider $Q_1$. According to the
second statement in Lemma \ref{lemma11}, we have
$\lambda_{n-1}\geq\min_{j=2,...,n}\{c_j^{-2}\}>0.$ Also, the
determinant of $Q_i$ is negative according to Lemma \ref{lemma12},
thus we have that $\lambda_n$ is negative. This proof is complete.
\end{proof}

\begin{lemma}
A lower bound and an upper bound of the eigenvalues of $Q_i$ given as (\ref{eq7}) are 
\begin{equation}\label{eq22}
\frac{1}{2}\left(\frac{n-2}{c_i^2}+\frac{1}{c_o^2}+\frac{1}{c_i}
\sqrt{\left(\frac{n-2}{c_i}-\frac{c_i}{c_o^2}\right)^2+4\sum_{j=1,j\neq
i}^n\frac{1}{c_j^2}}\right),
\end{equation}
and 
\begin{equation}\label{eq12}
\frac{1}{2}\left(\frac{n-2}{c_i^2}+\frac{1}{c_*^2}+\frac{1}{c_i}
\sqrt{\left(\frac{n-2}{c_i}-\frac{c_i}{c_*^2}\right)^2+4\sum_{j=1,j\neq
i}^n\frac{1}{c_j^2}}\right),
\end{equation}
respectively, where $c_o=\max_{j=1,j\neq i}^n\{c_j\}$ and $c_*=\min_{j=1,j\neq
i}^n\{c_j\}$.
% , which
%thereafter$\frac{1}{c_*^2}=\max_{j=1,j\neq i}^n\{c_j^{-2}\}$.
\end{lemma}
\begin{proof}
Without loss of generality,  we only consider $Q_1$ given as (\ref{eqq1}). In our proof, 
we let $Q_1$ be divided into two functions, $Q_1=f(x)+g(x), $ where $x\in
\mathbb{R}$, $f(x),$ and $g(x)$ are as follows:
\begin{equation}
f(x)=\frac{n-2-x}{c_1^2}E_{11}+\sum_{j=2}^n \frac{1}{c_j^2}E_{jj},
\text{ and }
g(x)=\frac{x}{c_1^2}E_{11}+\sum_{j=2}^n\frac{-1}{c_1c_j}(E_{ij}+E_{ji}).
\end{equation}
For any $x\in \mathbb{R},$ we have $\lambda_n(Q_1)\geq\lambda_n(f(x))+\lambda_n(g(x))$ and
$\lambda_1(Q_1)\leq \lambda_1(f(x))+\lambda_1(g(x)).$ Hence, 
\begin{equation}
\lambda_1(Q_1)\leq \min_{x\in
\mathbb{R}}\{\lambda_1(f(x))+\lambda_1(g(x))\},
\text{ and } \lambda_n(Q_1)\geq \max_{x\in
\mathbb{R}}\{\lambda_n(f(x))+\lambda_n(g(x))\}.
\end{equation}
Note that  $f(x)$ is a diagonal matrix for any $x\in \mathbb{R}$, then its eigenvalues are 
%\begin{equation}
$\lambda_1(f(x))=\max\{(n-2-x)c_1^{-2},c_*^{-2}\} \text { and } \lambda_n(f(x))=\min\{(n-2-x)c_1^{-2},c_o^{-2}\}.
$
%\end{equation}
Note that the rank of $g(x)$ is 1 for any $x\in \mathbb{R},$ and its characteristic polynomial is $
\lambda^{n-2}(\lambda^2-{x}c_1^{-2}\lambda-\sum_{j=2}^n(c_1c_j)^{-2}).
$ Thus, we have
\begin{equation}
\lambda_1(g(x))=\frac{1}{2}\left(\frac{x}{c_1^{2}}+\sqrt{\left(\frac{x}{c_1^{2}}\right)^2+4\sum_{j=2}^n\left(\frac{1}{c_1c_j}\right)^{2}}\right),
\end{equation}
and
\begin{equation}
\lambda_n(g(x))=\frac{1}{2}\left(\frac{x}{c_1^{2}}-\sqrt{\left(\frac{x}{c_1^{2}}\right)^2+4\sum_{j=2}^n\left(\frac{1}{c_1c_j}\right)^{2}}\right).
\end{equation}

We now consider the following four cases:

\begin{enumerate}
\item If $\frac{n-2-x}{c_1^2}\geq\frac{1}{c_*^2}$, {i.e.,} $x\leq
n-2-\frac{c_1^2}{c_*^2}$, then

\begin{equation}\label{c3}
\lambda_1(f(x))+\lambda_1(g(x))=\frac{n-2}{c_1^2}+\frac{2}{c_1}\left(\frac{\sum_{i=2}^n\frac{1}{c_i^2}}{\frac{x}{c_1}+\sqrt{\frac{x^2}{c_1^2}+4\sum_{i=2}^n\frac{1}{c_i^2}}}\right).
\end{equation}

  \item If $\frac{n-2-x}{c_1^2}\leq\frac{1}{c_*^2}$, {i.e.} $x\geq
n-2-\frac{c_1^2}{c_*^2}$, then
\begin{equation}\label{c4}
\lambda_1(f(x))+\lambda_1(g(x))=\frac{1}{c_*^2}+\frac{1}{2c_1}\left(\frac{x}{c_1}+\sqrt{\frac{x^2}{c_1^2}+4\sum_{i=2}^n\frac{1}{c_i^2}}\right).
\end{equation}

  \item If $\frac{n-2-x}{c_1^2}\geq\frac{1}{c_o^2}$, {i.e.,} $x\leq
n-2-\frac{c_1^2}{c_o^2}$, then
\begin{equation}\label{c1}
\lambda_n(f(x))+\lambda_n(g(x))=\frac{1}{c_o^2}-\frac{2}{c_1}\frac{\sum_{j=2}^n\frac{1}{c_j^2}}
{\frac{x}{c_1}+\sqrt{\frac{x^2}{c_1^2}+4\sum_{j=2}^n\frac{1}{c_j^2}}}.
\end{equation}

  \item If $\frac{n-2-x}{c_1^2}\leq\frac{1}{c_o^2}$, {i.e.,} $x\geq
n-2-\frac{c_1^2}{c_o^2}$, then 
\begin{equation}\label{c2}
\lambda_n(f(x))+\lambda_n(g(x))=\frac{n-2}{c_1^2}-\frac{1}{2c_1}\left(\frac{x}{c_1}+\sqrt{\frac{x^2}{c_1^2}+4\sum_{i=2}^n\frac{1}{c_i^2}}\right).
\end{equation}

\end{enumerate}
Since functions  (\ref{c1}) and (\ref{c4})
are increasing functions and  functions (\ref{c2}) and
(\ref{c3}) are decreasing functions, we have
\begin{equation}
arg \min_{x\in
\mathbb{R}}\{\lambda_1(f(x))+\lambda_1(g(x))\}=\{n-2-\frac{c_1^2}{c_*^2}\}
\end{equation}
and 
\begin{equation}
arg \max_{x\in
\mathbb{R}}\{\lambda_n(f(x))+\lambda_n(g(x))\}=\{n-2-\frac{c_1^2}{c_o^2}\}.
\end{equation}
Then substituting $x=n-2-\frac{c_1^2}{c_o^2}$ into (\ref{c1})
(or (\ref{c2})), and substituting $x=n-2-\frac{c_1^2}{c_*^2}$ into (\ref{c3})
(or (\ref{c4})), the lower and upper bounds (\ref{eq22}) and (\ref{eq12}) are immediate. 
This proof is complete. 
\end{proof}

\subsection{{The hyperplane is $\mathcal{H}=\mathcal{H}(e,
e^Tc)=\{x\in \mathbb{R}^n|e^Tx=c^Te\}.$}}

First of all, we compute an orthogonal
complementary basis of $e$. Obviously,
$\{e_1-e_2,e_1-e_3,...,e_1-e_n\}$ is a basis of the complementary
space of $e.$ We use this basis to construct an orthogonal complementary basis of $e$ given as follows:
\begin{equation}
h_i=-\frac{1}{\sqrt{i(i+1)}}\sum_{k=1}^ie_k+\frac{\sqrt{i}}{\sqrt{i+1}}e_{i+1},
\text{~~where~~} i=1,2,...,n-1,
\end{equation}
which can be explicitly written as  
\begin{equation}
H=[h_1,h_2,...,h_{n-1}]=\left[
                          \begin{array}{rrrr}
                            -\frac{1}{\sqrt{2}} & -\frac{1}{\sqrt{6}} & -\frac{1}{\sqrt{12}} & \cdots \\
                             \frac{1}{\sqrt{2}} & -\frac{1}{\sqrt{6}} & -\frac{1}{\sqrt{12}} & \cdots \\
                                                &  \frac{\sqrt{2}}{\sqrt{3}} & -\frac{1}{\sqrt{12}} & \cdots \\
                                                &                     & \frac{\sqrt{2}}{\sqrt{3}} & \cdots \\
                             &  &  & \ddots \\
                          \end{array}
                        \right].
\end{equation}
For simplicity, we denote the $i$th row of $H$ by $p_i$,
where $i=1,2,...,n.$ Without loss of generality, assume $i>j$, 
we have
\begin{equation}
p_i^Tp_j=-\frac{1}{{i}}+\sum_{k=i+1}^n\frac{1}{(k-1)k}=-\frac{1}{n},
\text{ and }
p_i^Tp_i=\frac{i-1}{{i}}+\sum_{k=i+1}^n\frac{1}{(k-1)k}=1-\frac{1}{n}.
\end{equation}
Thus, we have
\begin{equation}\label{eq23}
HH^T=I-\frac{1}{n}ee^T.
\end{equation}
Since $(HH^T)H=H$, {i.e.,} $(HH^T)h_i=h_i$, for $i\in\mathcal{I}(n-1),$ we have that $h_i$ is an eigenvector of $HH^T.$
Also, note that the rank of $HH^T$ is $n-1$, thus the eigenvalues of
$HH^T$ is 1 with multiplier $n-1$ and 0 with multiplier 1.

\begin{theorem}
Let a Dikin ellipsoid $\mathcal{E_D}$ be in the form of
(\ref{eq31}), and a hyperplane be $\mathcal{H}=\mathcal{H}(e,
e^Tc)$. Then the Lorenz cone $\mathcal{C_L}$ generated by the origin
and $\mathcal{H}\cap\mathcal{E_D}$ is represented as
$\mathcal{C_L}=\{x|x^TQx\leq0\}$, and
\begin{equation}\label{eq26}
Q=D+\frac{n-1}{(e^Tc)^2}ee^T-\frac{1}{e^Tc}(ec^TD+Dce^T)=D+\sum_{i=1}^n\sum_{j=1}^n\left(\frac{n-1}{(e^Tc)^2}-\frac{1}{e^Tc}\left(\frac{1}{c_i}+\frac{1}{c_j}\right)\right)E_{ij}.
\end{equation}
\end{theorem}

\begin{proof}
According to (\ref{eq23}), we have
$\tilde{H}=HH^T=I-\frac{1}{n}ee^T$, which yields $
D\tilde{H}=D-\frac{1}{n}Dee^T, $ and $\tilde{H}D=D-\frac{1}{n}ee^TD
$. Thus $
\tilde{D}=D-(D\tilde{H}+\tilde{H}D)=\frac{1}{n}(Dee^T+ee^TD)-D. $
Then we have
\begin{equation}\label{eq24}
1+c^T\tilde{D}c=1+\frac{1}{n}(c^TDee^Tc+c^Tee^TDc)-c^TDc=1-n+\frac{2\alpha}{n}\sum_{i=1}^n\frac{1}{c_i}.
\end{equation}
\begin{equation}\label{eq25}
\begin{split}
ec^TD\tilde{H}+\tilde{H}Dce^T&=ec^TD-\frac{1}{n}e(c^TDe)e^T+Dce^T-\frac{1}{n}e(e^TDc)e^T\\
                             &=ec^TD+Dce^T-\left(\frac{2}{n}\sum_{i=1}^n\frac{1}{c_i}\right)ee^T.
\end{split}
\end{equation}
Substituting (\ref{eq24}) and (\ref{eq25}) into (\ref{eq6}), this
lemma is immediate. The proof is complete.
\end{proof}

We now present an interesting result about the eigenvalue and
eigenvector structures of a class of rank 1 matrix.
\begin{lemma}
If the rank of a square matrix $A$ is $1$, the matrix $A$ has at
most $1$ nonzero eigenvalues with multiplier $1$.
\end{lemma}
\begin{proof}
The dimension of the kernel space of $Ax=0$ is $n-1.$ The basis in
this kernel space can be eigenvalues of $A$ corresponding to $0$.
Also, note that the sum of all eigenvalues is equal to the trace of
$A$ that may be zero, in which case, all eigenvalues are $0$. For
example, $A=E_{12}.$ The eigenvalues are all $0$ with eigenvectors
$e_1,e_2,...,e_n.$
\end{proof}

\begin{corollary}
Let two  vectors $a,c\in \mathbb{R}^n$, then the eigenvalues of
$ac^T$ are $a^Tc$ with multiplier $1$ and $0$ with multiplier $n-1$.
An eigenvector corresponding to $a^Tc$ is $a$. The eigenvectors
corresponding to $0$ can be the complementary basis of $c$.
\end{corollary}

\begin{corollary}
Let a nonzero vector $a\in \mathbb{R}^n$, then the matrix $aa^T$ is
rank $1$, and the eigenvalues of $aa^T$ are $\|a\|^2$ with
multiplier $1$ and $0$ with multiplier $n-1$. An eigenvector
corresponding to $\|a\|^2$ is $a$. The eigenvectors corresponding to
$0$ can be the complementary basis of $a$.
\end{corollary}

\begin{corollary}
The eigenvalues of $ee^T$ is $n$ with multiplier $1$ and eigenvector
$e$,  and $0$ with multiplier $n-1$ and eigenvectors
$e_1-e_2,e_1-e_3,...,e_1-e_n.$
\end{corollary}

\begin{lemma}\label{lemma13}
$\det(\beta I+\alpha ee^T)=(1+n\frac{\alpha}{\beta})\beta^n.$
\end{lemma}
\begin{proof}
Let $T_n=\det(\beta I+\alpha ee^T).$ Writing $T_n$ by definition, we
can find the following iteration formula:
\begin{equation}
T_n=(\beta+\alpha)T_{n-1}+\sum_{j=2}^n (-1)^{1+j}\alpha
M_{1j}=(\beta+\alpha)T_{n-1}-(n-1)\alpha^2\beta^{n-2}.
\end{equation}
Let $F_n=\frac{T_n}{\beta^n},$ then
$F_n=\frac{\beta+\alpha}{\beta}F_{n-1}-(n-1)\frac{\alpha^2}{\beta^2}.$
It is easy to prove that $F_n-F_{n-1}=\frac{\alpha}{\beta}.$ Then
this lemma is immediate.
\end{proof}

\begin{corollary}
$\det(I+\alpha ee^T)=1+n\alpha.$
\end{corollary}

\begin{lemma}
If $c_1=c_2=...=c_n=c,$ then eigenvalues  of $Q$ in (\ref{eq26}) is
$\frac{1}{c^2}$ with multiplier $n-1$ and $-\frac{1}{nc^2}$ with
multiplier 1.
\end{lemma}
\begin{proof}
Since $c_1=c_2=...=c_n=c,$ we have
$Q=\frac{1}{c^2}(I-\frac{n+1}{n}ee^T).$ According to Lemma
\ref{lemma13}, we have $\det(Q-\lambda I )=0$ to yield the
characteristics polynomial is $(1-\lambda
c^2)^{n-1}(-\frac{1}{n}-\lambda c^2)=0.$ The the lemma is immediate.
\end{proof}

%\begin{lemma}
%What is all eigenvalues of the $Q$?
%\end{lemma}

\subsection{\textbf{Tangent cone}}
\begin{lemma}\label{lemma15}
Let $A,B,C,D\in \mathbb{R}^n.$ Assume $AC\bot CB, CD\bot AB$, and
$D\in AB$. Then $ \|AB\|^2=\|AC\|^2+\|BC\|^2,
~~\|AC\|^2=\|AD\|\|AB\|, $ and $ \|BC\|^2=\|BD\|\|AB\|. $
\end{lemma}
\begin{proof}
Note that
$\overrightarrow{AB}=\overrightarrow{CB}-\overrightarrow{CA}$ and
$\overrightarrow{CB}\bot\overrightarrow{CA}$, the first equation is
immediate. The remaining two are not trivial to prove, even they
looks natural in 2 or 3 dimension. Now we present a rigorous proof.

We denote $x_A,x_B,x_C,$ and $x_D$ the coordinates of $A,B,C$ and
$D$ in $\mathbb{R}^n$, respectively. Since $D\in AB$, there exists
$0\leq \lambda \leq 1$ such that $x_D=\lambda x_A+(1-\lambda)x_B.$
Since $AC\bot BC $ and $CD\bot AB$, we have
\begin{equation}\label{eq27}
\left(\lambda x_A+(1-\lambda)x_B-x_C\right)^T(x_A-x_B)=0,~~~~~
(x_A-x_C)^T(x_B-x_C)=0.
\end{equation}
Expanding two equations in (\ref{eq27}) and plugging
$x_A^Tx_B-x_C^Tx_A=x_B^Tx_C-\|x_C\|^2$, which is obtained from the
second equation in (\ref{eq27}), into the first equation in
(\ref{eq27}), we have
\begin{equation}
\lambda=\frac{\|x_B-x_C\|^2}{\|x_A-x_B\|^2}=\frac{\|\overrightarrow{BC}\|^2}{\|\overrightarrow{AB}\|^2}.
\end{equation}
Also, $DA=(1-\lambda)(x_B-x_A)$ and $BA=x_B-x_A$, then
\begin{equation}
\|AD\|\|AB\|=(1-\lambda)(x_B-x_A)^T(x_B-x_A)=(1-\lambda)\|x_A-x_B\|^2=\|\overrightarrow{AB}\|^2-\|\overrightarrow{BC}\|^2=\|\overrightarrow{AC}\|^2.
\end{equation}
Similarly, the third equation in this lemma is easy to obtain.
\end{proof}

\begin{lemma}\label{lemma16}
The distance of a point $\bar{x}\in \mathbb{R}^n$ to a hyperplane
$\mathcal{S}=\{a^Tx=\alpha\}$ is
\begin{equation}
dist(x,S)=\frac{|a^T\bar{x}-\alpha|}{\|a\|}.
\end{equation}
\end{lemma}

\begin{theorem}
Assume the ellipsoid $\mathcal{E}=\{x\in \mathbb{R}^n|
(x-c)^T(x-c)\leq 1\},$ then the Lorenz cone $\mathcal{C}$ is
\begin{equation}
Q=I-\frac{1}{\|c\|^2-1}cc^T.
\end{equation}
\end{theorem}
\begin{proof}
Since the ellipsoid $\mathcal{E}$ is a $n$-sphere, the hyperplane
through the intersection of the ellipsoid and Lorenz cone (note that
this intersection is a $n-1$-dimensional ellipsoid) is $\{x\in
\mathbb{R}^n|c^Tx=\alpha\}$. We now compute the value of $\alpha.$
According to Lemma \ref{lemma15}, we obtain that the distance from
origin to this hyperplane is $\|c\|-1/\|c\|.$ Also, according to
Lemma \ref{lemma16}, we can compute $\alpha=\|c\|^2-1.$ Thus, the
hyperplane through $\mathcal{E}\cap \mathcal{C}$ is
\begin{equation}
\mathcal{S}=\{x\in \mathbb{R}^n~|~ c^Tx=\|c\|^2-1\}.
\end{equation}
According to \cite{julio1}, there exist parameters $z\in
\mathbb{R}^{n-1}, \lambda, \mu$, such that
\begin{equation}\label{eq28}
D=I+\lambda cc^T+cz^TH^T+Hzc^T, ~~0=-c-\mu c-(\|c\|^2-1)Hz,
\end{equation}
\begin{equation}
0=\|c\|^2-1-\lambda(\|c\|^2-1)^2+2\mu(\|c\|^2-1).
\end{equation}
From the second and third equation in (\ref{eq28}), we have
\begin{equation}\label{eq29}
Hz=\frac{-1}{\|c\|^2-1}(1+\mu)c, ~~~ \lambda(\|c\|^2-1)-2\mu =1.
\end{equation}
Substitute (\ref{eq29}) into the first equation in (\ref{eq28}), the
theorem is immediate.
\end{proof}

%Note: this matrix seems  called Householder or elementary refection
%matrix. which may hold the properties that $H^T=H, H^2=I, H^{-1}=H$.
%See http://www.math.umn.edu/~olver/num\_/lnqr.pdf.

\section{Invariance of Lorenz Cone}

 A tangent cone of a convex set $\mathcal{C}\subset \mathbb{R}^n$ at point $x\in
\mathcal{C}$ is defined as
\begin{equation}
T_{\mathcal{C}}(x)=\{z\in \mathbb{R}^n|\lim_{h\rightarrow0}\inf\frac{dist(x+hz,\mathcal{C})}{h}=0\}
\end{equation}
where $dist(x,\mathcal{C})=\inf_{y\in \mathcal{C}}\|x-y\|$. It is easy to see that
$T_{\mathcal{C}}(x)=\mathbb{R}^n$ if $x$ is located interior of $\mathcal{C}$. If $x$ is on the
boundary of $\mathcal{C}$ with smooth neighborhood, then the tangent cone is
the affine half-space which is obtained by parallelling the tangent
line at $x$ with respect to $\mathcal{C}$ to be through origin.

\begin{lemma}\cite{song1, nagu}
\textbf{[Nagumo]} Let $\mathcal{C}$ be convex and closed set in $R^n$, then
$\mathcal{C}$ is invariant with respect to dynamical system
(\ref{eqn:dy1}) if and only if $\forall x\in \partial \mathcal{C}$, then
$Ax\in T_{\mathcal{C}}(x)$, where $T_{\mathcal{C}}(x)$ is tangent cone of $\mathcal{C}$ at $x$.
\end{lemma}
This is an elegant and intuitive conclusion. We can understand this lemma
in the following way: the right side, i.e. $Ax$ in dynamical system
is actually the slope of tangent line of trajectory at point $x$,
since the left side is the derivative of trajectory. If the
trajectory starts from one point in $\mathcal{C}$, then the unique possibility
to move out from this set must be through some point on the
boundary. This lemma states that the slope is in the tangent cone at
$x$, which forces the trajectory to move back to $\mathcal{C}$.

Based on Nagumo lemma, it is easy to derive the sufficient and
necessary condition that one ellipsoid or cone is positively
invariant with respect to dynamical system (\ref{eqn:dy1}). For simplicity, we  remove the suffix $c$ in $A_c$ in  (\ref{eqn:dy1}). 
\begin{lemma}
Let an ellipsoid be defined as $\mathcal{E}=\{x\in \mathbb{R}^n|x^TPx\leq1\}$, an cone
be defined as $\mathcal{C}=\{x\in \mathbb{R}^n|x^TQx\leq0\}$, then $\mathcal{E}$ or $\mathcal{C}$ is
a positively invariant set with respect to dynamical system if and
only if $\langle Ax,Px\rangle\leq0$ or $\langle Ax,Qx\rangle\leq0$
for $x$ on the boundary of the set, where $\langle a, b\rangle$ is
the inner product of $a$ and $b$.
\end{lemma}
\begin{proof} We just prove the ellipsoid case, the proof for cone
is almost same. It is easy to see that the outer normal at $x\in
\partial \mathcal{E}$ is $2Px$, and since the boundary of an ellipsoid is
smooth, the tangent cone at $x$ is
\begin{equation}
T_{\mathcal{E}}(x)=\{y|\langle y, Px\rangle\leq0\}.
\end{equation}
Then by Nagumo lemma, dynamical system (\ref{eqn:dy1}) is positively
invariant on $\mathcal{E}$ if and only if $Ax\in T_{\mathcal{E}}(x)$ for any $x$ on the boundary of
$\mathcal{E}$. Thus, this lemma was proved.
\end{proof}

According to \cite{stern}, it concluded that the dynamical system
(\ref{eqn:dy1}) is positively invariant for the ice cream cone
$\mathcal{C}_0=\{x\in R^n|x_1^2+...+x_{n-1}^2\leq x_n^2\}$ if and only if
there exists $a\in R$ such that
\begin{equation}
Q_nA+A^TQ_n+aQ_n\leq0,
\end{equation}
where $Q_n=diag(1,...,1,-1)$, and the inequality means semi-negative
definite.
\begin{lemma}
Let cone be defined as $\mathcal{C}=\{x\in \mathbb{R}^n|x^TQx\leq0\}$, then dynamical
system is positively invariant on $\mathcal{C}$ if and only if there exists
$a\in \mathbb{R}$ such that
\begin{equation}\label{eqn:a}
QA+A^TQ+aQ\leq0,
\end{equation}
where the inequality means semi-negative definite.
\end{lemma}

\begin{proof} 
There exists one nonsingular transformation $P$ such that $\mathcal{C}=P\mathcal{C}_0$. Thus, $\forall x\in \mathcal{C}_0$,
there exists $x^*\in \mathcal{C}$ such that $x^*=Px$. Since $x^*$ satisfies
dynamical system equation, we have
\begin{equation}
(x^*)^{'}=Ax^*\Leftrightarrow (Px)^{'}=APx\Leftrightarrow
x^{'}=P^{-1}APx.
\end{equation}
Thus, dynamical system (\ref{eqn:dy1}) is positively invariant on
$\mathcal{C}$ is equivalent that the right dynamical system is positive
invariant on $\mathcal{C}_0$. By the previous lemma and $P^{T}QP=Q_n$, there
exist $a\in R$ such that
\begin{equation}
\begin{split}
Q_nP^{-1}AP+(P^{-1}AP)^TQ_n+aQ_n&\leq0\\
P^{T}QPP^{-1}AP+P^TA^TPP^{-1}QP+aP^{T}QP&\leq0\\
P^T(QA+A^TQ+aQ)P&\leq0
\end{split}
\end{equation}
In fact, the last ``inequality" is equivalent with (\ref{eqn:a}). If
(\ref{eqn:a}) is true, then for any $x$,
\begin{equation}
x^TP^T(QA+A^TQ+aQ)Px=(Px)^T(QA+A^TQ+aQ)(Px)\leq0.
\end{equation}
The the other hand, for any $x$, since $P$ is singular, there exists
a $y$ such that $P^{-1}x=y$, i.e. $x=Py$. Then
\begin{equation}
x^T(QA+A^TQ+aQ)x=y^TP^T(QA+A^TQ+aQ)Py\leq0.
\end{equation}
The proof is completed. 
\end{proof}

Figure \ref{fig511} gives the shape of the constructed Lorenz cone by using the Dikin ellipsoid and two different hyperplanes. 
\begin{figure}[h]
    \centering
       \includegraphics[width=0.4\textwidth]{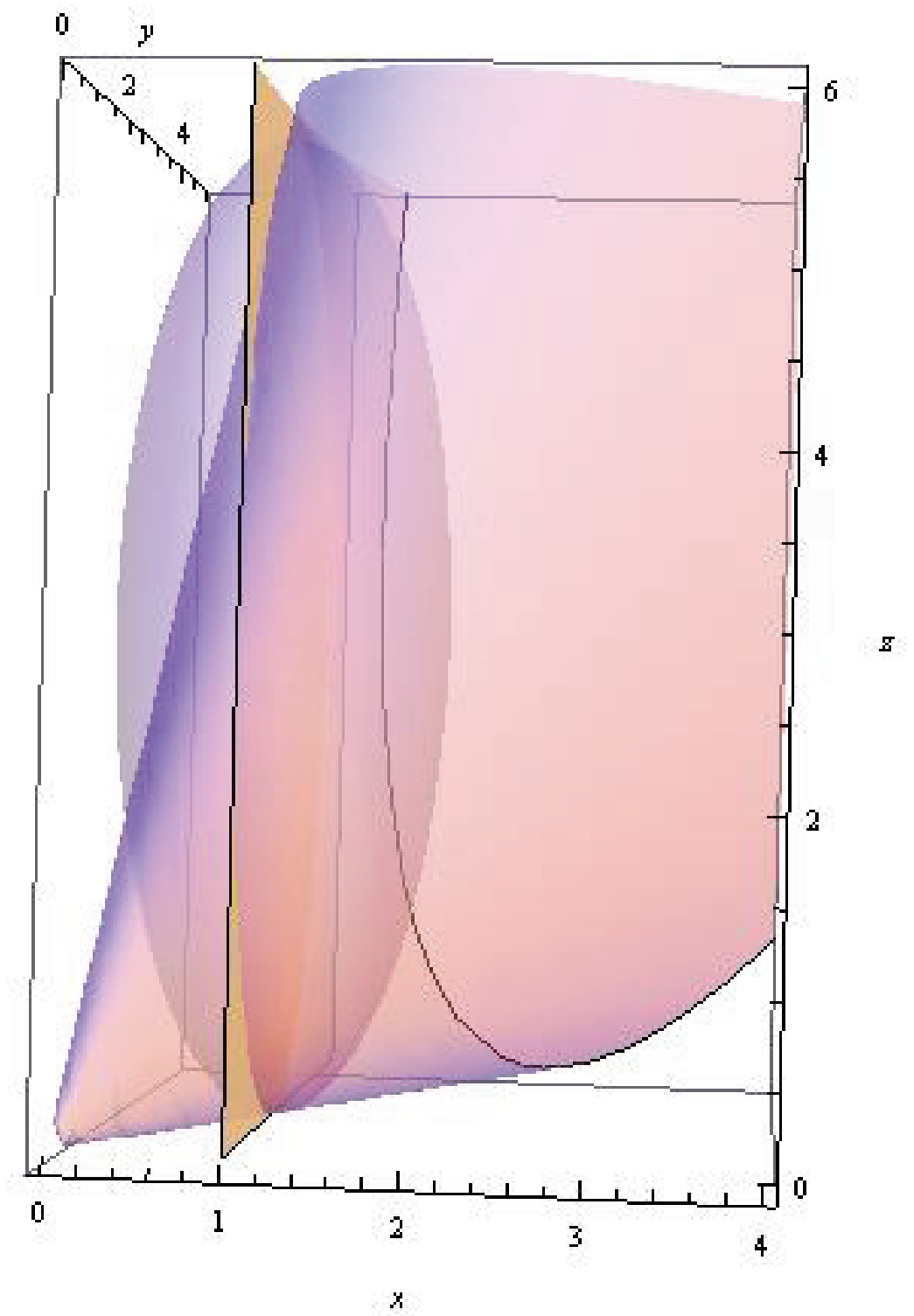}
       \includegraphics[width=0.47\textwidth]{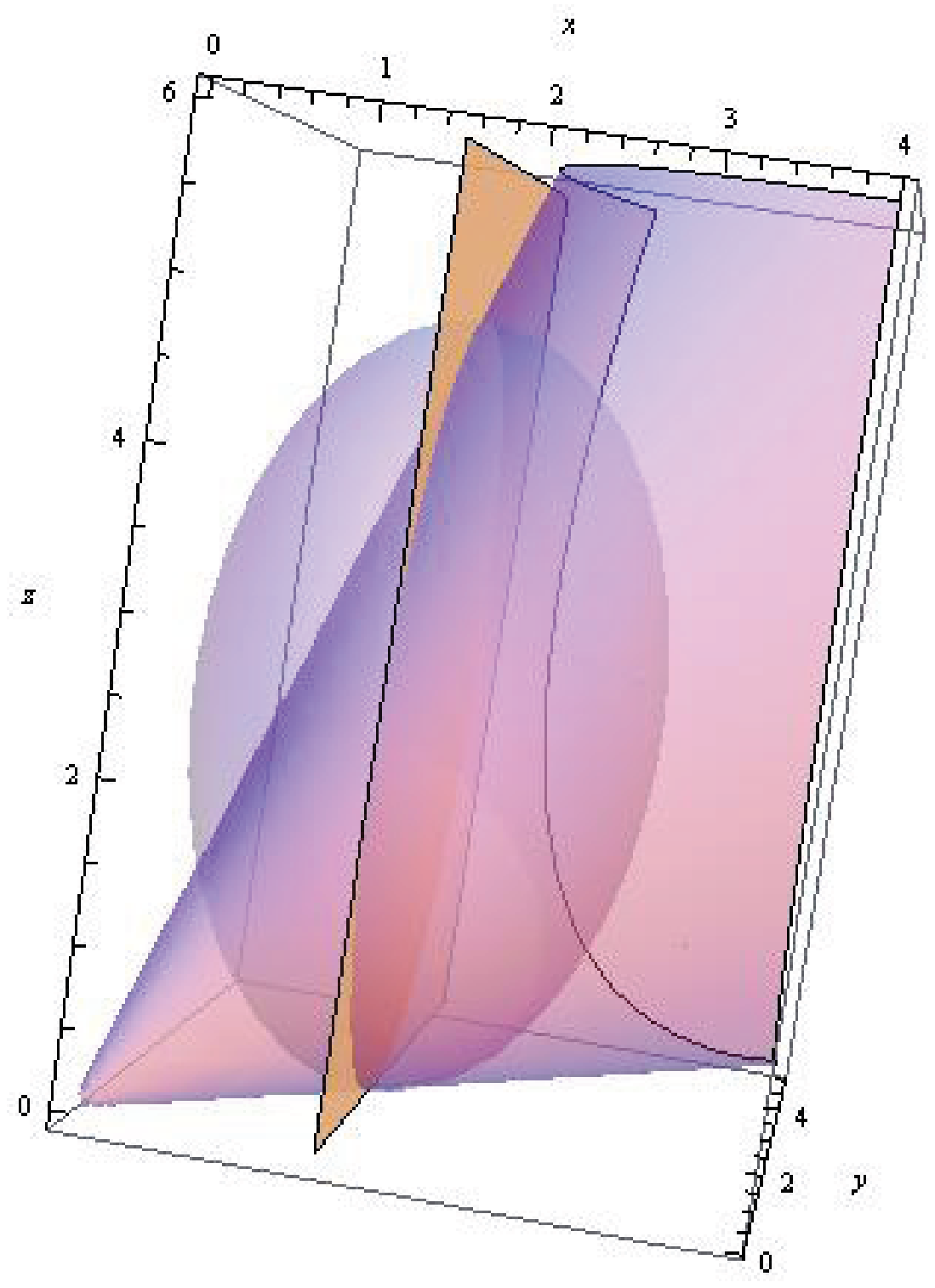}
       \caption{A framework to derive invariance conditions for continuous systems.}
       \label{fig511}
\end{figure}

\section{Conclusion}
In this paper, we study the cone based invariant sets for a given dynamical system. We construct some special Lorenz cones using Dikin ellipsoid and some hyperplanes. Dikin  ellipsoid is originally introduced from mathematical optimization to design the  polynomial time linear algorithm.  We use this tool into the construction to ensure that the constructed Lorenz cone is located in positive orthant, which is usually a common requirement in the real world application. We also study the structure of the constructed cones, especially the eigenvalues structure of the related matrix in the formula of elliposid. The novelty of this paper is building the link between optimization and invariant sets. It also provides more flexibility for other researchers either in theory or in practice to choose more Lorenz cones for analysis.

\bibliographystyle{plain}
\bibliography{myref}

\end{document}